\documentclass[12pt]{amsart}

\usepackage{amsmath, amssymb, graphics, setspace}
\usepackage[usenames]{color}
\input xypic

\usepackage{algorithm}
\usepackage[noend]{algpseudocode}

\algblock{Input}{EndInput}
\algnotext{EndInput}
\algblock{Output}{EndOutput}
\algnotext{EndOutput}

\makeatletter
\def\BState{\State\hskip-\ALG@thistlm}
\makeatother

\usepackage{verbatim}
\usepackage{url}
\usepackage{bm}
\usepackage{tikz}
\usepackage{algorithm}
\usepackage[noend]{algpseudocode}

\newcommand{\mathsym}[1]{{}}
\newcommand{\unicode}[1]{{}}

\makeatletter
\def\BState{\State\hskip-\ALG@thistlm}
\makeatother

\usepackage{amsmath}
\usepackage[centertags]{amsmath}
\usepackage{amsfonts}
\usepackage{amssymb}
\usepackage{amsthm}
\usepackage{newlfont}
\usepackage{url}
\usepackage{bm}
\usepackage{chngcntr}
\usepackage{tikz}
\theoremstyle{plain}
\newtheorem{thm}{Theorem}

\newtheorem{lem}[thm]{Lemma}
\newtheorem{lem*}[thm]{Lemma}
\newtheorem{prop}[thm]{Proposition}

\theoremstyle{definition}

\theoremstyle{remark}
\newtheorem{rem}{Remark}
\newtheorem{rem*}{Remark}

\newtheorem{example}[rem]{Example}

\numberwithin{rem}{section} 
\numberwithin{dfn}{section} 
\numberwithin{equation}{section} 
\numberwithin{thm}{section} 

\def\!{\operatorname{!}}

\def\R{\mathbb R}

\def\1{\bold 1}

\def\deg{\operatorname{deg}}

\def\deg{\operatorname{deg}}

\usepackage{amssymb}

\usepackage{enumitem}

\usepackage{graphicx}
\usepackage{xy}
\usepackage{float}

\makeatletter
\@namedef{subjclassname@2020}{%
	\textup{2020} Mathematics Subject Classification}
\makeatother

\usepackage[T1]{fontenc}

\newtheorem{theorem}{Theorem}[section]

\newtheorem{lemma}[theorem]{Lemma}

\theoremstyle{definition}
\newtheorem{definition}[theorem]{Definition}
\newtheorem{remark}[theorem]{Remark}

\numberwithin{equation}{section}

\title[ Banded Reduction of Integrals... ]{\textbf{Banded Reduction of Integrals for Hyperexponential and Algebraic Functions}}
\author{Piotr Kraso{\'n} \and Jan Milewski}
\date{\today}

	\address{Department of Mathematics, Ariel University, Ariel 40700, Israel}
	\email{ piotrkras26@gmail.com}
    
    \address{  Institute of Mathematics\\ 
    Pozna{\'n} University of Technology\\
    60-965 Pozna{\'n}\\
    Poland
    }
    \email{jsmilew@wp.pl}

\begin{document}

\maketitle

\begin{abstract}
We develop a finite-band method for the reduction of families of
indefinite integrals. The starting point is the case in which the
logarithmic derivative $K'(x)/K(x)$ is a rational function. This leads
to adapted polynomial and Laurent-type bases and to upper triangular
band matrices. The framework includes generalized real
Schwarz--Christoffel integrals and more general hyperexponential
weights. We then extend the construction to the case in which the
logarithmic derivative $K'(x)/K(x)$ is algebraic over the field of real
rational functions. The associated finite-dimensional differential
module gives rise naturally to block-band systems. Finally, we show
that functions satisfying linear differential equations with real
polynomial coefficients give rise to finite recurrences for their
moments. The resulting reduction is effective: for a prescribed target
integral it is enough to solve a finite triangular band system, and only
the reduction data associated with the target are required. We also
discuss arithmetic complexity and give explicit examples illustrating
the scalar and block reductions. The algorithms are implemented in the
accompanying Mathematica package \texttt{BandedReduction}.
\end{abstract}

\section{Introduction}

Reduction of integrals by differential relations is a classical theme
in symbolic integration and in the algorithmic study of special
functions. Classical approaches include Hermite-type reduction for
algebraic functions and the symbolic integration of hyperexponential
terms; see \cite{bronstein1,bronstein2,trager}. A closely related
algorithmic viewpoint is provided by holonomic functions and creative
telescoping, where differential equations and recurrences are used to
control families of functions and definite or parameter-dependent
integrals; see \cite{az,zeilberger,chyzak}.

The purpose of this paper is to develop a finite-band approach to the
reduction of families of indefinite integrals. Our basic objects are the
moments
\[
I_n(x)=\int x^nK(x)\,dx,
\qquad n\geq0,
\]
or, more generally, the shifted moments
\[
I_{n,p}(x)=\int (x-p)^nK(x)\,dx.
\]
When the factor $K$ satisfies an appropriate differential relation,
differentiation of suitable multiples of $K$ produces linear relations
involving only finitely many neighbouring moments. Ordering these
relations by the moment index leads to a banded triangular system. A
prescribed high-order moment can therefore be reduced to a finite
family of basic integrals together with explicit boundary terms.

The reduction mechanism is related to the classical Christoffel-type
approach to the reduction of moments. Our purpose is not to replace
this underlying principle, but to identify and exploit the finite-band
structure that arises for hyperexponential and algebraic integrands.
This structure makes it possible to determine explicitly the finite
range of indices involved in the reduction and leads to an effective
finite-dimensional reduction procedure.

Our first setting is the case in which the logarithmic derivative
\[
\frac{K'(x)}{K(x)}
\]
is a rational function. This is a standard class in symbolic
integration and includes, in particular, functions of the form
\[
K(x)=e^{q(x)}\prod_{j=1}^{s}p_j(x)^{\alpha_j},
\]
where $q$ and the $p_j$ are real polynomials and
$\alpha_j\in\mathbb R$. We work on intervals on which the chosen real
branches of the factors $p_j(x)^{\alpha_j}$ are defined and $K$ is
nonzero. Such functions are closely related to the hyperexponential
framework of symbolic integration \cite{bronstein1,bronstein2}.
Relations between hypergeometric functions and integral
representations have also been extensively studied; see, for example,
\cite{bry,exton,kniehl,kratt,vidunas}.

A principal family of examples comes from the Schwarz--Christoffel
setting. Schwarz--Christoffel transformations provide a classical and
computationally useful method for conformal mapping of polygonal
domains; see, for example, \cite{driscoll,trefethen}. The associated
integral representations are closely related to multivariable
hypergeometric functions of Lauricella type \cite{lauricella}.
The generalized real Schwarz--Christoffel integrals considered in this
paper also belong to a broader class of problems related to
hyperelliptic integrals; see \cite{km}.

The present paper should be viewed as a continuation and substantial
generalization of the reduction approach developed in \cite{km}. The
earlier work treats a particular class of real hyperelliptic integrals
and derives recurrence relations adapted to that setting. Here we
identify the differential mechanism underlying those recurrences and
formulate it in a more general finite-band framework. The change of
viewpoint is structural rather than merely notational: starting from
differential relations satisfied by the integrand factor $K$, we obtain
a unified reduction procedure that applies to rational logarithmic
derivatives, algebraic logarithmic derivatives, and, more generally,
linear differential equations with polynomial coefficients. The
generalized real Schwarz--Christoffel integrals are retained as an
important concrete specialization of this broader construction.

The first main part of the paper establishes a finite-band reduction
scheme for rational logarithmic derivatives. We construct adapted
polynomial and Laurent-type bases, identify the corresponding upper
triangular band matrices, and isolate the possible resonant indices.
In the nonresonant case this yields a finite family of basic integrals
to which the whole family can be reduced.

We next consider the case in which

$$
\frac{K'}{K}
$$
is algebraic over $\R(x)$. Classical symbolic integration of algebraic
functions is based on algebraic reduction procedures, including
Hermite-type methods; see \cite{trager,malykh}. Our construction has a
different emphasis: the moments are organized into a finite-dimensional
differential module, and the scalar band matrices are replaced by
block-band matrices.

Finally, we consider arbitrary linear differential equations with
real polynomial coefficients,
\[
a_r(x)K^{(r)}+\cdots+a_0(x)K=0.
\]

This places the preceding constructions in a broader holonomic
framework \cite{az,chyzak,zeilberger}. The essential point is not merely
the existence of a recurrence, but its finite-band structure and the
resulting effective reduction of a prescribed moment.

The computational aspect of the method is equally important. For a
fixed target index one does not need an inverse of an infinite matrix;
only a finite triangular truncation is involved, and in particular only
the data required for the target reduction need to be computed.
Standard backward substitution for triangular linear systems is used
at this stage \cite{quarteroni,stoer}. Explicit formulas for inverses of
triangular matrices are also closely related to the computations
appearing here \cite{bd}.

The algorithms developed in the paper are implemented in the accompanying
Mathematica package \emph{BandedReduction: Banded Reduction of Integrals for
Hyperexponential and Algebraic Functions} \cite{BandedReduction}, referred to
below simply as \texttt{BandedReduction}. The package provides a user-level
implementation of the scalar, Schwarz--Christoffel, algebraic block-band, and
polynomial-ODE reductions discussed below, together with validation tests and
readable reduction reports. The package is archived on Zenodo under DOI
10.5281/zenodo.22980629.

The paper is organized as follows. Section~2 develops the scalar
first-order construction. Section~3 applies it to generalized real
Schwarz--Christoffel integrals. Section~4 treats algebraic logarithmic
derivatives and the resulting block-band systems. Section~5 derives
finite moment recurrences from polynomial-coefficient differential
equations. Section~6 describes the reduction algorithm, discusses its
arithmetic complexity, and presents explicit examples.

Throughout the paper, $\mathbb R[x]$ denotes the ring of real
polynomials in the variable $x$, and $\mathbb R(x)$ denotes its field
of rational functions. For $A,B\in\mathbb R[x]$, the condition
$\gcd(A,B)=1$ means that $A$ and $B$ are relatively prime. We write
$M_d(F)$ for the ring of $d\times d$ matrices over a field $F$.

\section{A first-order differential reduction scheme}

Let $J\subset\mathbb R$ be an interval and let
$K:J\to\mathbb R\setminus\{0\}$ be a differentiable function. We assume
throughout this section that
\begin{equation}
\frac{K'(x)}{K(x)}
=
\frac{A(x)}{B(x)},
\qquad
A,B\in\mathbb R[x],\qquad \gcd(A,B)=1.
\label{eq:2.1}
\end{equation}
Thus the logarithmic derivative of $K$ is rational. This class includes,
in particular, functions of the form
\[
K(x)=e^{q(x)}
\prod_{j=1}^{s}p_j(x)^{\alpha_j},
\qquad
q,p_j\in\mathbb R[x],\quad \alpha_j\in\mathbb R,
\]
on any interval on which the chosen real branches of the factors
$p_j(x)^{\alpha_j}$ are defined and $K$ is nonzero. Such functions
belong to the standard hyperexponential framework of symbolic
integration; see \cite{bronstein1,bronstein2}.

Fix $p\in\mathbb R$ such that
\[
B(p)\neq0
\]
and put
\[
u=x-p.
\]
For $n\in\mathbb Z$ define
\[
I_{n,p}(x)=\int u^nK(x)\,dx.
\]
Set
\[
W(x)=B'(x)+A(x).
\]

\begin{lem}
For every $n\in\mathbb Z$,
\begin{equation}
\frac{d}{dx}\left(u^nB(x)K(x)\right)
=
\left[
nu^{n-1}B(x)+u^nW(x)
\right]K(x).
\label{eq:2.2}
\end{equation}
\end{lem}

\begin{proof}
Since $u'=1$ and $K'/K=A/B$,
\[
\begin{aligned}
\frac{d}{dx}\left(u^nBK\right)
&=
nu^{n-1}BK
+u^n(B'K+BK')\\
&=
\left[
nu^{n-1}B+u^n(B'+A)
\right]K.
\end{aligned}
\]
\end{proof}

Write
\begin{equation}
B(x)=\sum_{j=0}^{b}c_j u^j,
\qquad
W(x)=\sum_{j=0}^{w}d_j u^j,
\label{eq:2.3}
\end{equation}
and put
\begin{equation}
h=\max\{b,w+1\}.
\label{eq:2.4}
\end{equation}

\begin{prop}
For every $n\in\mathbb Z$,
\begin{equation}
\sum_{j=0}^{b}nc_jI_{n+j-1,p}
+
\sum_{j=0}^{w}d_jI_{n+j,p}
=
u^nB(x)K(x)+C.
\label{eq:2.5}
\end{equation}
Consequently, the family $\{I_{n,p}\}_{n\in\mathbb Z}$ satisfies a
finite recurrence in the index $n$, with coefficients affine in $n$.
\end{prop}

\begin{proof}
Substituting \eqref{eq:2.3} into \eqref{eq:2.2} and integrating gives
\eqref{eq:2.5}.
\end{proof}

\subsection{The associated band matrix}

Set
\[
m=n+h-1.
\]
Then \eqref{eq:2.5} becomes
\begin{equation}
\sum_{l=m-h}^{m}D_{l,m}I_{l,p}
=
u^{m-h+1}B(x)K(x)+C,
\label{eq:2.6}
\end{equation}
where
\begin{equation}
D_{l,m}
=
(m-h+1)c_{l-m+h}
+d_{l-m+h-1}.
\label{eq:2.7}
\end{equation}
We understand $c_j=d_j=0$ whenever the index is outside the ranges
specified in \eqref{eq:2.3}.
next consider
\begin{prop}
The matrix
\[
D=(D_{l,m})_{l,m\geq0}
\]
is upper triangular and has finite bandwidth. More precisely,
\[
D_{l,m}=0
\qquad\text{if}\qquad
l<m-h\quad\text{or}\quad l>m.
\]
Its diagonal entries are
\begin{equation}
D_{m,m}
=
(m-h+1)c_h+d_{h-1}.
\label{eq:2.8}
\end{equation}
\end{prop}

\begin{proof}
The asserted vanishing follows immediately from \eqref{eq:2.7}.
Putting $l=m$ gives \eqref{eq:2.8}.
\end{proof}

\begin{definition}
For $m\geq h-1$, we call $m$ \emph{admissible}. An admissible index is
called \emph{nonresonant} if
\[
D_{m,m}\neq0,
\]
and \emph{resonant} otherwise.
\end{definition}

\begin{remark}
The diagonal coefficient
\[
D_{m,m}=(m-h+1)c_h+d_{h-1}
\]
is an affine function of $m$. Hence, unless it is identically zero, there
is at most one resonant admissible index.
\end{remark}

\subsection{Adapted polynomial basis}

For the initial indices
\[
0\leq m\leq h-1,
\]
put
\[
\phi_m(u)=u^m.
\]
For every $m\geq h$, define
\begin{equation}
\phi_m(u)
=
\sum_{l=m-h}^{m}D_{l,m}u^l.
\label{eq:2.9}
\end{equation}

Since $m\geq h$, all powers occurring in \eqref{eq:2.9} are
nonnegative. Thus $\phi_m$ is a polynomial.

By \eqref{eq:2.6},
\begin{equation}
\int\phi_m(u)K(x)\,dx
=
u^{m-h+1}B(x)K(x)+C,
\qquad m\geq h.
\label{eq:2.10}
\end{equation}

\begin{lem}
Let $m\geq h$. If $m$ is nonresonant, then
\[
\deg\phi_m=m.
\]
\end{lem}

\begin{proof}
All terms in \eqref{eq:2.9} have degree at most $m$, and the coefficient
of $u^m$ is $D_{m,m}$. If $m$ is nonresonant, this coefficient is
nonzero, and therefore $\deg\phi_m=m$.
\end{proof}

\begin{prop}
Suppose that
\[
D_{m,m}\neq0
\qquad\text{for all }m\geq h.
\]
Then the family
\[
\{\phi_m\}_{m\geq0}
\]
is a basis of $\mathbb R[u]$.
\end{prop}

\begin{proof}
For $0\leq m\leq h-1$ the functions $\phi_m$ are the monomials.
For $m\geq h$, the preceding lemma gives $\deg\phi_m=m$. Hence the
transition matrix from the monomial basis
\[
\{1,u,u^2,\ldots\}
\]
to $\{\phi_m\}_{m\geq0}$ is triangular with nonzero diagonal on every
finite principal truncation. Therefore it is invertible on every finite
degree subspace, and the family is a basis of $\mathbb R[u]$.
\end{proof}

\begin{remark}
If a resonant index occurs, the corresponding adapted polynomial may
have degree smaller than its index. The reduction at that index requires
a separate treatment; in particular, one should not divide by the
vanishing diagonal coefficient.
\end{remark}

\subsection{Negative powers}

For $n\leq-2$, differentiate
\[
u^{n+1}B(x)K(x).
\]
We obtain
\[
\frac{d}{dx}\left(u^{n+1}BK\right)
=
\left[
(n+1)u^nB+u^{n+1}W
\right]K.
\]

Define
\begin{equation}
\psi_n(u)
=
\sum_{l=n}^{n+h}
\left[
(n+1)c_{l-n}
+d_{l-n-1}
\right]u^l,
\qquad n\leq-2.
\label{eq:2.11}
\end{equation}
Then
\begin{equation}
\int\psi_n(u)K(x)\,dx
=
u^{n+1}B(x)K(x)+C.
\label{eq:2.12}
\end{equation}

\begin{prop}
Assume $B(p)\neq0$. The transition matrix from the standard Laurent
monomials to the family $\{\psi_n\}_{n\leq-2}$ is triangular and banded.
Its diagonal entry corresponding to $n$ is
\[
(n+1)B(p),
\]
and therefore is nonzero for every $n\leq-2.$
\end{prop}

\begin{proof}
The coefficient of the lowest power $u^n$ in \eqref{eq:2.11} is
\[
(n+1)c_0=(n+1)B(p).
\]
Since $n\leq-2$ and $B(p)\neq0$, this coefficient is nonzero.
\end{proof}

\subsection{Finite-band reduction theorem}

\begin{thm}[Finite-band reduction]
Let $K$ satisfy \eqref{eq:2.1}, and let $p\in\mathbb R$ satisfy
$B(p)\neq0$. Suppose that
\[
D_{m,m}\neq0
\qquad\text{for all }m\geq h.
\]
Then every integral
\[
I_{n,p}=\int(x-p)^nK(x)\,dx,
\qquad n\in\mathbb Z,
\]
can be reduced to the finite family
\begin{equation}
\mathcal B_p=
\left\{
I_{-1,p},I_0,\ldots,I_{h-1,p}
\right\}
\label{eq:2.13}
\end{equation}
and explicit boundary terms.
\end{thm}

\begin{proof}
For $n\leq-2$, the triangular band system associated with
\eqref{eq:2.12} has nonzero diagonal and reduces every negative-index
integral to $I_{-1,p}$ and finitely many nonnegative-index integrals.

For $n\geq0$, the adapted polynomial basis above gives a triangular
band system whose diagonal entries are nonzero by assumption. Hence
every $I_{n,p}$ with $n\geq h$ reduces recursively to
\[
I_0,\ldots,I_{h-1}
\]
and the boundary terms in \eqref{eq:2.10}.
Combining the two sectors gives the stated reduction.
\end{proof}

\subsection{Finite computation of reduction coefficients}

For a fixed target index $M$, let
\[
\mathcal D^{[M]}
=
(D_{l,m})_{0\leq l,m\leq M}
\]
be the finite principal truncation of $D$. Since $\mathcal D^{[M]}$ is
upper triangular,
\[
\det\mathcal D^{[M]}
=
\prod_{m=0}^{M}D_{m,m}.
\]
Consequently, $\mathcal D^{[M]}$ is nonsingular if and only if all its
diagonal entries are nonzero.

Only the column corresponding to the target integral is required. If
$v^{(M)}$ denotes the corresponding column of
\[
\left(\mathcal D^{[M]}\right)^{-1},
\]
then
\[
\mathcal D^{[M]}v^{(M)}=e_M.
\]
The triangular and banded structure gives
\begin{equation}
v^{(M)}_M=\frac{1}{D_{M,M}},
\label{eq:2.14}
\end{equation}
and, proceeding backwards,
\begin{equation}
v^{(M)}_i
=
-\frac{1}{D_{i,i}}
\sum_{j=i+1}^{\min(i+h,M)}
D_{i,j}v^{(M)}_j.
\label{eq:2.15}
\end{equation}

Thus no inverse of an infinite matrix is required. For fixed bandwidth,
the number of arithmetic operations required for one reduction column
is linear in the size of the truncation, apart from the cost of
arithmetic in the coefficient field.

\subsection{Example 1: a hyperexponential weight}

Consider
\begin{equation}
K(x)=e^x(x^2+1)^\alpha,
\qquad \alpha\in\mathbb R.
\label{eq:2.16}
\end{equation}
Then
\[
\frac{K'(x)}{K(x)}
=
1+\frac{2\alpha x}{x^2+1}
=
\frac{x^2+2\alpha x+1}{x^2+1}.
\]
Thus
\[
A(x)=x^2+2\alpha x+1,
\qquad
B(x)=x^2+1,
\]
and
\[
W(x)=B'(x)+A(x)
=
x^2+2(\alpha+1)x+1.
\]
For $p=0$ the fundamental identity gives
\[
\frac{d}{dx}
\left(x^n(x^2+1)K(x)\right)
=
\left[
nx^{n-1}
+x^n
+(n+2\alpha+2)x^{n+1}
+x^{n+2}
\right]K(x).
\]
Hence
\begin{equation}
nI_{n-1}+I_n+(n+2\alpha+2)I_{n+1}+I_{n+2}
=
x^n(x^2+1)K(x)+C.
\label{eq:2.17}
\end{equation}
Here $h=3$ and the diagonal coefficient is
\[
D_{m,m}=1.
\]
Thus the positive-index sector is nonresonant and every moment reduces
to the basic family
\[
I_{-1},\ I_0,\ I_1,\ I_2.
\]

\begin{remark}
This example illustrates an important feature of the reduction
procedure. The recurrence expresses an infinite family of integrals in
terms of a finite set of basic integrals and explicit boundary terms.
The basic integrals need not themselves be elementary. If they are
elementary, then the recurrence produces elementary antiderivatives for
the whole family. If some of them are non-elementary, the same
recurrence still reduces every integral in the family to a fixed finite
collection of non-elementary basic integrals. Thus finite-dimensional
reduction and elementary integrability are separate issues: the former
may hold even when the latter fails.
\end{remark}

\subsection{Example 2: a shorter recurrence}

Consider
\begin{equation}
K(x)=e^{-x^2}(x^2+1)^\alpha,
\qquad \alpha\in\mathbb R.
\label{eq:2.18}
\end{equation}
Then
\[
\frac{K'(x)}{K(x)}
=
\frac{-2x^3+2(\alpha-1)x}{x^2+1},
\]
so that
\[
A(x)=-2x^3+2(\alpha-1)x,
\qquad
B(x)=x^2+1,
\]
and
\[
W(x)=-2x^3+2\alpha x.
\]
Consequently,
\begin{equation}
nI_{n-1}
+
(n+2\alpha)I_{n+1}
-
2I_{n+3}
=
x^n(x^2+1)K(x)+C.
\label{eq:2.19}
\end{equation}
Only one parity class occurs in each recurrence. This illustrates that
the actual recurrence may possess additional structure beyond the
formal band bound of the general construction.

\section{Generalized real Schwarz--Christoffel integrals}

We now apply the first-order reduction scheme of Section~2 to a class
of generalized real Schwarz--Christoffel integrals. This provides an
explicit and important specialization of the general construction.

Let
\begin{equation}
K(x)=
\prod_{i=1}^{N_1}(x-x_i)^{\alpha_i}
\prod_{i=N_1+1}^{N}(1-x_i x)^{\beta_i},
\qquad
\alpha_i,\beta_i\in\mathbb R,
\qquad
|\alpha_i|<1,\quad |\beta_i|<1,
\label{eq:3.1}
\end{equation}
and let the interval under consideration be chosen so that $K$ is real
and nonzero. We consider integrals
\begin{equation}
\int R(x)K(x)\,dx,
\label{eq:3.2}
\end{equation}
where $R(x)$ is rational. By polynomial division and partial fraction
decomposition, it is enough to consider the positive-index family
\begin{equation}
I_n(x)=\int x^nK(x)\,dx,
\qquad n\geq0,
\label{eq:3.3}
\end{equation}
and the negative-index families
\begin{equation}
I_{n,p}(x)=\int (x-p)^nK(x)\,dx,
\qquad n<0,
\label{eq:3.4}
\end{equation}
where $p$ is chosen away from the zeros of the polynomial factor
introduced below.

The connection with the classical Schwarz--Christoffel and Lauricella
settings is described in
\cite{driscoll,trefethen,lauricella}.

\subsection{The polynomial factor}

Set
\begin{equation}
P(x)=
\prod_{i=1}^{N_1}(x-x_i)
\prod_{i=N_1+1}^{N}(1-x_i x).
\label{eq:3.5}
\end{equation}
Then
\[
K_1(x)=P(x)K(x)
\]
and, provided
\[
\lambda:=(-1)^{N-N_1}
\prod_{i=N_1+1}^{N}x_i\neq0,
\]
we have
\[
\deg P=N.
\]

Moreover,
\begin{equation}
\frac{K'(x)}{K(x)}
=
\sum_{i=1}^{N_1}\frac{\alpha_i}{x-x_i}
-
\sum_{i=N_1+1}^{N}
\frac{\beta_i x_i}{1-x_i x}.
\label{eq:3.6}
\end{equation}
Thus, in the notation of Section~2, we may take
\begin{equation}
B(x)=P(x),
\qquad
A(x)=P(x)\frac{K'(x)}{K(x)}.
\label{eq:3.7}
\end{equation}
Explicitly,
\begin{equation}
A(x)=
\sum_{i=1}^{N_1}
\alpha_i\frac{P(x)}{x-x_i}
-
\sum_{i=N_1+1}^{N}
\beta_i x_i\frac{P(x)}{1-x_i x}.
\label{eq:3.8}
\end{equation}
Hence $A(x)$ is a polynomial.

The polynomial occurring in the differential identity is
\begin{equation}
W(x)=B'(x)+A(x)
=
P'(x)+P(x)\frac{K'(x)}{K(x)}.
\label{eq:3.9}
\end{equation}
Equivalently,
\begin{equation}
W(x)=
\left[
\sum_{i=1}^{N_1}
\frac{\alpha_i+1}{x-x_i}
-
\sum_{i=N_1+1}^{N}
\frac{x_i(\beta_i+1)}{1-x_i x}
\right]P(x),
\label{eq:3.10}
\end{equation}
and therefore $W(x)$ is a polynomial of degree at most $N-1$.

Put
\[
\lambda=(-1)^{N-N_1}
\prod_{i=N_1+1}^{N}x_i,
\]
which is the leading coefficient $\operatorname{lc}(P)$ of $P$.
Then
\begin{equation}
\operatorname{lc}(W)
=
\lambda
\left(
N+\sum_{i=1}^{N_1}\alpha_i
+\sum_{i=N_1+1}^{N}\beta_i
\right).
\label{eq:3.11}
\end{equation}

\subsection{The fundamental identity}

Fix $p\in\mathbb R$ such that
\[
P(p)\neq0
\]
and put
\[
u=x-p.
\]
The fundamental identity of Section~2 becomes
\begin{equation}
\frac{d}{dx}\left(u^nP(x)K(x)\right)
=
\left[
nu^{n-1}P(x)+u^nW(x)
\right]K(x).
\label{eq:3.12}
\end{equation}

Write
\[
P(x)=\sum_{j=0}^{N}a_j u^j,
\qquad
W(x)=\sum_{j=0}^{N-1}b_j u^j.
\label{eq:3.13}
\]
Then
\[
a_N=\lambda,
\qquad
b_{N-1}
=
\lambda
\left(
N+\sum_{i=1}^{N_1}\alpha_i
+\sum_{i=N_1+1}^{N}\beta_i
\right).
\label{eq:3.14}
\]
Consequently, in the notation of Section~2,
\[
h=N.
\]

\subsection{Positive powers}

For the positive-index sector we put $p=0$, so that $u=x$, and write
\[
I_n=\int x^nK(x)\,dx.
\]
For every $m\geq N$, define
\begin{equation}
\phi_m(x)
=
\sum_{l=m-N}^{m}
\left[
(m-N+1)a_{l-m+N}
+
b_{l-m+N-1}
\right]x^l.
\label{eq:3.15}
\end{equation}
For the initial indices
\[
0\leq m\leq N-1,
\]
put
\[
\phi_m(x)=x^m.
\]

By the differential identity,
\begin{equation}
\int\phi_m(x)K(x)\,dx
=
x^{m-N+1}P(x)K(x)+C,
\qquad m\geq N.
\label{eq:3.16}
\end{equation}

The coefficient of $x^m$ in $\phi_m$ is
\begin{equation}
\delta_m
=
(m-N+1)a_N+b_{N-1}.
\label{eq:3.17}
\end{equation}
Using \eqref{eq:3.14}, this becomes
\begin{equation}
\delta_m
=
\lambda
\left[
m+1+
\sum_{i=1}^{N_1}\alpha_i+
\sum_{i=N_1+1}^{N}\beta_i
\right].
\label{eq:3.18}
\end{equation}

Thus an index $m\geq N$ is nonresonant whenever
\begin{equation}
m+1+
\sum_{i=1}^{N_1}\alpha_i+
\sum_{i=N_1+1}^{N}\beta_i
\neq0.
\label{eq:3.19}
\end{equation}
Since the left-hand side is affine in $m$, there is at most one
resonant index.

\begin{prop}
Assume that \eqref{eq:3.19} holds for every $m\geq N$. Then
\[
\{\phi_m\}_{m\geq0}
\]
is a basis of $\mathbb R[x]$, and every positive-index integral
$I_m$, $m\geq N$, is reducible to
\[
I_0,\ldots,I_{N-1}
\]
and explicit boundary terms.
\end{prop}

\begin{proof}
For $0\leq m\leq N-1$, the functions $\phi_m$ are the monomials.
For $m\geq N$, condition \eqref{eq:3.19} gives
\[
\deg\phi_m=m.
\]
Hence the transition matrix from the monomial basis to the adapted
basis is triangular with nonzero diagonal on every finite principal
truncation. The reduction follows from \eqref{eq:3.16}.
\end{proof}

\subsection{The transition matrix}

Let
\[
\mathcal A=(\mathcal A_{l,m})_{l,m\geq0}
\]
denote the infinite transition matrix from the monomial basis
\[
\gamma_m(x)=x^m
\]
to the adapted basis
\[
\{\phi_m(x)\}_{m\geq0}.
\]
Then
\begin{equation}
\mathcal A_{l,m}
=
\begin{cases}
\delta_{l,m},
&0\leq m\leq N-1,\\[1mm]
0,
&l>m\ \text{or}\ l<m-N,\\[1mm]
(m-N+1)a_{l-m+N}
+b_{l-m+N-1},
&m\geq N.
\end{cases}
\label{eq:3.20}
\end{equation}
Thus $\mathcal A$ is upper triangular and banded.

For a fixed $M\geq0$, let
\[
\mathcal A^{[M]}
=
(\mathcal A_{l,m})_{0\leq l,m\leq M}
\]
be its finite principal truncation. Since $\mathcal A^{[M]}$ is upper
triangular,
\[
\det\mathcal A^{[M]}
=
\prod_{m=0}^{M}\mathcal A_{m,m}.
\]
Consequently,
\[
\mathcal A^{[M]}
\text{ is nonsingular}
\iff
\mathcal A_{m,m}\neq0
\quad\text{for all }0\leq m\leq M.
\]

For $M\geq N$, the finite truncation has the block form
\[
\mathcal A^{[M]}
=
\left(
\begin{array}{cc}
I & E^{[M]}\\[2mm]
0 & D^{[M]}
\end{array}
\right),
\]
where $D^{[M]}$ is upper triangular and banded. Hence
\begin{equation}
\left(\mathcal A^{[M]}\right)^{-1}
=
\left(
\begin{array}{cc}
I &
-\,E^{[M]}\left(D^{[M]}\right)^{-1}\\[2mm]
0 &
\left(D^{[M]}\right)^{-1}
\end{array}
\right).
\label{eq:3.21}
\end{equation}
For the reduction of a fixed target integral, only the corresponding
column of $\left(\mathcal A^{[M]}\right)^{-1}$ is required.

\subsection{Negative powers}

Let $p\in\mathbb R$ satisfy
\[
P(p)\neq0.
\]
For $n\leq-2$, define
\begin{equation}
\psi_n(u)
=
\sum_{l=n}^{n+N}
\left[
(n+1)a_{l-n}
+
b_{l-n-1}
\right]u^l.
\label{eq:3.22}
\end{equation}
Then
\begin{equation}
\int\psi_n(u)K(x)\,dx
=
u^{n+1}P(x)K(x)+C.
\label{eq:3.23}
\end{equation}

The coefficient of the lowest power $u^n$ is
\[
(n+1)a_0=(n+1)P(p),
\]
and is therefore nonzero for every $n\leq-2$. Hence the corresponding
transition matrix is triangular and banded with nonzero diagonal.

Consequently, all negative-index integrals reduce to
\[
I_{-1,p}
\]
and finitely many nonnegative-index integrals. In the nonresonant case,
the natural basic family is
\begin{equation}
\mathcal B_p=
\left\{
I_{-1,p},I_0,\ldots,I_{N-1,p}
\right\}.
\label{eq:3.24}
\end{equation}

\subsection{Explicit example}

Consider
\begin{equation}
K(x)=
(x-1)^{-1/3}(x+1)^{-2/3}
(x-2)^{-2/5}(x+3)^{-3/4}(x-4)^{-1/3}.
\label{eq:3.25}
\end{equation}
Then $N=5$ and
\[
K_1(x)=
(x-1)^{2/3}(x+1)^{1/3}
(x-2)^{3/5}(x+3)^{1/4}(x-4)^{2/3}.
\]

For $n=4$ the construction gives
\begin{align}
\int x^4K(x)\,dx
={}&
-\frac{1032}{151}\int K(x)\,dx
-\frac{1326}{151}\int xK(x)\,dx
+\frac{1303}{151}\int x^2K(x)\,dx
\nonumber\\
&+
\frac{246}{151}\int x^3K(x)\,dx
+\frac{60}{151}K_1(x)+C.
\label{eq:3.26}
\end{align}

For $n=5$,
\begin{align}
\int x^5K(x)\,dx
={}&
-\frac{222192}{31861}\int K(x)\,dx
-\frac{811308}{31861}\int xK(x)\,dx
+\frac{110232}{31861}\int x^2K(x)\,dx
\nonumber\\
&+
\frac{401209}{31861}\int x^3K(x)\,dx
+\frac{25560}{31861}K_1(x)
+\frac{60}{211}xK_1(x)+C.
\label{eq:3.27}
\end{align}

For the negative-index sector, choose $p=-2$. Then
\begin{align}
\int (x+2)^{-2}K(x)\,dx
={}&
\frac{91}{4320}\int(x+2)^3K(x)\,dx
-\frac{337}{2160}\int(x+2)^2K(x)\,dx
\nonumber\\
&+
\frac{617}{4320}\int(x+2)K(x)\,dx
+\frac{1147}{2160}\int K(x)\,dx
\nonumber\\
&+
\frac{11}{60}\int(x+2)^{-1}K(x)\,dx
-\frac1{72}(x+2)^{-2}K_1(x)+C.
\label{eq:3.28}
\end{align}

After expanding
\[
(x+2)^3=x^3+6x^2+12x+8,
\qquad
(x+2)^2=x^2+4x+4,
\]
and collecting terms, we obtain
\begin{align}
\int (x+2)^{-2}K(x)\,dx
={}&
\frac{91}{4320}\int x^3K(x)\,dx
-\frac4{135}\int x^2K(x)\,dx
-\frac{329}{1440}\int xK(x)\,dx
\nonumber\\
&+
\frac{13}{36}\int K(x)\,dx
+\frac{11}{60}\int(x+2)^{-1}K(x)\,dx
-\frac1{72}(x+2)^{-2}K_1(x)+C.
\label{eq:3.29}
\end{align}

\section{Algebraic logarithmic derivatives}

We now extend the scalar reduction scheme of Section~2 to the case in
which the logarithmic derivative of the function $K$ is algebraic over
$\mathbb R(x)$. The point of passing to an algebraic extension is that
an algebraic logarithmic derivative can be represented in a
finite-dimensional vector space over $\mathbb R(x)$. The scalar moment
relations are then replaced naturally by vector relations and
block-band matrices.

Let
\[
L/\mathbb R(x)
\]
be a finite algebraic extension of degree
\[
d=[L:\mathbb R(x)]<\infty,
\]
and assume that $K$ is a nonzero differentiable function on an interval
$J\subset\mathbb R$ such that
\begin{equation}
\rho(x):=\frac{K'(x)}{K(x)}\in L.
\label{eq:4.1}
\end{equation}

Choose a basis
\[
v_1,\ldots,v_d
\]
of $L$ over $\mathbb R(x)$ and put
\[
\mathbf v(x)=
\left(
\begin{array}{c}
v_1(x)\\
\vdots\\
v_d(x)
\end{array}
\right).
\]

Since the derivation \(d/dx\) on \(\mathbb R(x)\) extends uniquely to the finite algebraic extension \(L\), we have
\[
\frac{d}{dx}L\subseteq L.
\]
Consequently, for the chosen basis \(v_1,\ldots,v_d\) of \(L\) over \(\mathbb R(x)\), there exists a matrix
\[
M(x)\in M_d(\mathbb R(x))
\]
such that
\begin{equation}
\mathbf v'(x)=M(x)\mathbf v(x).
\label{eq:4.2}
\end{equation}
Multiplication by $\rho$ is an $\mathbb R(x)$-linear endomorphism of
$L$. Hence there exists a matrix
\[
R(x)\in M_d(\mathbb R(x))
\]
such that
\begin{equation}
\rho(x)\mathbf v(x)=R(x)\mathbf v(x).
\label{eq:4.3}
\end{equation}
Therefore
\begin{equation}
\frac{d}{dx}\bigl(\mathbf v(x)K(x)\bigr)
=
H(x)\mathbf v(x)K(x),
\qquad
H(x)=M(x)+R(x).
\label{eq:4.4}
\end{equation}

Choose a nonzero polynomial
\[
q(x)\in\mathbb R[x]
\]
which clears the denominators of the entries of $H(x)$, and set
\[
G(x)=q(x)H(x)\in M_d(\mathbb R[x]),
\qquad
Q(x)=q'(x)I_d+G(x).
\]

\begin{prop}
For every $n\geq0$,
\begin{equation}
\frac{d}{dx}
\left(
x^nq(x)\mathbf v(x)K(x)
\right)
=
\left[
nx^{n-1}q(x)I_d+x^nQ(x)
\right]\mathbf v(x)K(x).
\label{eq:4.5}
\end{equation}
\end{prop}

\begin{proof}
The identity follows immediately from the product rule and
\eqref{eq:4.4}.
\end{proof}

Write
\[
q(x)=\sum_{j=0}^{b}q_jx^j,
\qquad
Q(x)=\sum_{j=0}^{w}Q_jx^j,
\]
where $Q_j\in M_d(\mathbb R)$, and define
\[
\mathbf I_n(x)
=
\int x^n\mathbf v(x)K(x)\,dx.
\]

\begin{thm}[Block-band reduction]
Under the assumptions above, the vector moments satisfy the finite
block recurrence
\begin{equation}
\sum_{j=0}^{b}
nq_j\mathbf I_{n+j-1}
+
\sum_{j=0}^{w}
Q_j\mathbf I_{n+j}
=
x^nq(x)\mathbf v(x)K(x)+\mathbf C.
\label{eq:4.6}
\end{equation}
Consequently, after ordering the relations according to their largest
moment index, one obtains a finite block-band system with blocks of
size $d\times d$.
\end{thm}

\begin{proof}
Expanding \eqref{eq:4.5} and integrating with respect to $x$ gives
\eqref{eq:4.6}. Only finitely many shifts occur, so the resulting
system is block-banded. The dependence on $n$ is affine.
\end{proof}

Let $\mathcal D_m$ denote the diagonal block corresponding to the
highest-index vector $\mathbf I_m$ in the triangularly ordered system.

\begin{prop}
Suppose that
\[
\det\mathcal D_m
\]
is not identically zero as a polynomial in $m$. Then only finitely many
indices are resonant, and for every nonresonant index the corresponding
relation can be solved for $\mathbf I_m$.
\end{prop}

\begin{proof}
The determinant $\det\mathcal D_m$ is a nonzero polynomial in $m$ and
therefore has only finitely many zeros. At every nonresonant index the
diagonal block is invertible, which allows the recurrence to be solved
for the highest-index vector moment.
\end{proof}

\subsection{Resonant blocks and partial elimination}

At a resonant index the diagonal block need not be invertible, but this
does not mean that the whole vector moment has to be added to the basic
family.  After the lower-index moments have already been reduced, the
block row at index $m$ can be written in the form
\begin{equation}
\mathcal D_m\mathbf I_m=\mathbf R_m,
\label{eq:4.6a}
\end{equation}
where $\mathbf R_m$ contains only lower-index vector moments and explicit
boundary terms.

Assume that
\[
\operatorname{rank}\mathcal D_m=r<d.
\]
Row reduction of $\mathcal D_m$ separates the columns into a set of pivot
columns $P_m$ and a set of free columns $F_m$.  The components of
$\mathbf I_m$ corresponding to $P_m$ can then be expressed in terms of
$\mathbf R_m$ and the components corresponding to $F_m$.  Thus a
singular diagonal block leaves only the free components unresolved at
that stage.

\begin{prop}
Let $\mathcal D_m$ be a singular diagonal block of rank $r$.  The block
row \eqref{eq:4.6a} leaves at most $d-r$ independent components of
$\mathbf I_m$ undetermined.  Equivalently, after row reduction, all
pivot components can be eliminated in favour of lower-index moments,
explicit boundary terms, and the free components.
\end{prop}

\begin{proof}
Apply elementary row operations to \eqref{eq:4.6a}.  Since row
operations preserve the solution set, the reduced system has $r$ pivot
columns and $d-r$ free columns.  Each pivot variable is determined by
one pivot equation in terms of the right-hand side and the free
variables.  Hence only the components indexed by the free columns may
remain undetermined.
\end{proof}

The block-band structure provides additional information that is not
visible in the single row \eqref{eq:4.6a}.  A free component may occur
with nonzero coefficient in one of the neighbouring block rows.  We
therefore use the following elimination rule at a resonant index:
first determine the pivot and free columns of $\mathcal D_m$; next
eliminate all pivot components; then inspect the finitely many
neighbouring block relations and use them, whenever possible, to
eliminate the remaining free components.  Only those independent free
components which cannot be removed by this neighbouring elimination
are retained as additional basic integrals.

This procedure is the block analogue of retaining a resonant scalar
moment in the basic family, but it can be strictly more economical:
a singular $d\times d$ block need not contribute all $d$ components,
and neighbouring relations may reduce the number of additional basic
integrals below $d-r$.

\subsection{The algebraic family $y^d=q(x)$}

We now consider the particularly transparent family
\begin{equation}
y^d=q(x),
\qquad
\frac{K'(x)}{K(x)}=y,
\label{eq:4.7}
\end{equation}
where $q(x)\in\mathbb R[x]$ and the extension generated by $y$ has
degree $d$ over $\mathbb R(x)$. Symbolic integration of algebraic
functions is classically based on Hermite-type and related reduction
procedures; see \cite{trager,bronstein2,malykh}. Our approach differs
in that we organize the corresponding moments into a finite block-band
system.

Differentiating $y^d=q(x)$ gives
\begin{equation}
y'=\frac{q'(x)}{d\,q(x)}\,y.
\label{eq:4.8}
\end{equation}
Hence
\[
V=
\operatorname{span}_{\mathbb R(x)}
\{1,y,\ldots,y^{d-1}\}
\]
is stable under differentiation.

Define
\[
I_n^{(j)}
=
\int x^ny^jK(x)\,dx,
\qquad
0\leq j\leq d-1.
\]
For $0\leq j\leq d-2$,
\begin{equation}
\frac{d}{dx}
\left(x^ny^jK(x)\right)
=
\left(
nx^{n-1}y^j
+
\frac{j}{d}x^ny^j\frac{q'(x)}{q(x)}
+
x^ny^{j+1}
\right)K(x).
\label{eq:4.9}
\end{equation}
After multiplication by $q(x)$, all coefficients are polynomial.
For $j=d-1$, the relation closes using
\[
y^d=q(x).
\]
Thus the families
\[
\{I_n^{(j)}\}_{n\geq0},
\qquad
0\leq j\leq d-1,
\]
satisfy a finite block-band recurrence.

\subsection{The model case $y^d=x$}

Consider
\begin{equation}
y^d=x,
\qquad
\frac{K'(x)}{K(x)}=y.
\label{eq:4.10}
\end{equation}
Then
\[
y'=\frac{1}{dx}y.
\]
For
\[
I_n^{(j)}=\int x^ny^jK(x)\,dx
\]
we obtain, for $0\leq j\leq d-2$,
\begin{equation}
\left(n+\frac{j}{d}\right)I_{n-1}^{(j)}
+
I_n^{(j+1)}
=
x^ny^jK(x)+C,
\label{eq:4.11}
\end{equation}
while
\begin{equation}
\left(n+\frac{d-1}{d}\right)I_{n-1}^{(d-1)}
+
I_{n+1}^{(0)}
=
x^ny^{d-1}K(x)+C.
\label{eq:4.12}
\end{equation}
These relations form a finite cyclic block system.

\begin{example}
Consider the quadratic extension
\[
y^2=x
\]
and choose
\[
\frac{K'(x)}{K(x)}=y=\sqrt{x}.
\]
On $x>0$ we may take
\[
K(x)=\exp\left(\frac23x^{3/2}\right).
\]
Define
\[
I_n=\int x^nK(x)\,dx,
\qquad
J_n=\int x^nyK(x)\,dx.
\]
Then
\begin{align}
nI_{n-1}+J_n
&=
x^nK(x)+C,
\label{eq:4.13}\\
\left(n+\frac12\right)J_{n-1}+I_{n+1}
&=
x^nyK(x)+C.
\label{eq:4.14}
\end{align}
Using \eqref{eq:4.13} with $n$ replaced by $n-1$, we obtain
\[
J_{n-1}
=
x^{n-1}K(x)-(n-1)I_{n-2}.
\]
Substitution into \eqref{eq:4.14} gives
\begin{equation}
I_{n+1}
=
\left(n+\frac12\right)(n-1)I_{n-2}
+
x^nyK(x)
-
\left(n+\frac12\right)x^{n-1}K(x)
+C.
\label{eq:4.15}
\end{equation}
Thus the auxiliary family $\{J_n\}$ can be eliminated and the scalar
moments satisfy a finite-band recurrence.  This is an explicit instance
of the neighbouring elimination described above: although the natural
block formulation is singular, a component not determined by one block
row is eliminated using the adjacent relation.  In particular, the moments
are recursively determined by the initial integrals
\[
I_0,\qquad I_1,\qquad I_2.
\]
\end{example}

\section{Reduction from polynomial-coefficient differential equations}

The first-order construction of Section~2 admits a broader formulation
which does not require the logarithmic derivative of $K$ to be rational.
We now assume that $K$ satisfies a linear differential equation with
polynomial coefficients.  This places the reduction procedure in the
standard holonomic framework; see, for example, \cite{az,chyzak,zeilberger}.

Suppose that
\begin{equation}
a_r(x)K^{(r)}(x)+a_{r-1}(x)K^{(r-1)}(x)
+\cdots+a_1(x)K'(x)+a_0(x)K(x)=0,
\label{eq:5.1}
\end{equation}
where
\[
a_j(x)\in\mathbb R[x],
\qquad
a_r(x)\not\equiv0.
\]
Write
\begin{equation}
a_j(x)=\sum_{s=0}^{d_j}a_{j,s}x^s,
\qquad
d_j=\deg a_j.
\label{eq:5.2}
\end{equation}

For $n\geq0$, define
\[
I_n(x)=\int x^nK(x)\,dx.
\]

We first record the repeated integration-by-parts identity used below.

\begin{lemma}
For $j\geq1$ and $s\geq0$,
\begin{align}
\int x^{n+s}K^{(j)}(x)\,dx
={}&
\sum_{t=0}^{j-1}
(-1)^t(n+s)_{\underline t}
x^{n+s-t}K^{(j-1-t)}(x)
\nonumber\\
&+
(-1)^j
(n+s)_{\underline j}
I_{n+s-j}(x),
\label{eq:5.3}
\end{align}
where
\[
(t)_{\underline j}
=
t(t-1)\cdots(t-j+1),
\qquad
(t)_{\underline0}=1.
\]
\end{lemma}

\begin{proof}
The identity follows by repeated integration by parts.  For $j=1$,
\[
\int x^{n+s}K'(x)\,dx
=
x^{n+s}K(x)
-
(n+s)I_{n+s-1}(x).
\]
The general case follows by induction on $j$.
\end{proof}

\subsection{Moment recurrence}

The passage from linear differential equations to recurrences for
integral families is related to the holonomic approach to symbolic
summation and integration; see, for example, \cite{chyzak}.  Our
construction emphasizes the resulting finite-band structure and its
use for the reduction of moments to a finite basic family.

\begin{thm}[Moment recurrence]
Under the assumptions \eqref{eq:5.1}, the moments $I_n$ satisfy a finite
recurrence
\begin{equation}
\sum_k C_k(n)I_{n+k}(x)=E_n(x),
\label{eq:5.4}
\end{equation}
where $C_k(n)\in\mathbb R[n]$ and $E_n(x)$ is an explicit linear
combination of
\[
x^mK(x),\,
x^mK'(x),\,
\ldots,\,
x^mK^{(r-1)}(x).
\]
More precisely,
\begin{equation}
C_k(n)
=
\sum_{\substack{0\leq j\leq r\\0\leq s\leq d_j\\s-j=k}}
(-1)^j
a_{j,s}(n+s)_{\underline j}.
\label{eq:5.5}
\end{equation}
Only finitely many coefficients $C_k$ are nonzero.  Their possible
shifts satisfy
\begin{equation}
-r\leq k\leq
\max_{0\leq j\leq r}(d_j-j).
\label{eq:5.6}
\end{equation}
\end{thm}

\begin{proof}
Multiply \eqref{eq:5.1} by $x^n$ and integrate.  For each term
$a_{j,s}x^{n+s}K^{(j)}(x)$ with $j\geq1$, apply Lemma~5.1.  The terms
containing derivatives of $K$ of order at most $r-1$ form the explicit
boundary expression $E_n(x)$.  The remaining integral term is
\[
(-1)^j a_{j,s}(n+s)_{\underline j}I_{n+s-j}(x).
\]
Grouping terms with the same shift
\[
k=s-j
\]
gives \eqref{eq:5.4} and \eqref{eq:5.5}.  Since
\[
0\leq j\leq r,
\qquad
0\leq s\leq d_j,
\]
the possible shifts satisfy \eqref{eq:5.6}.
\end{proof}

\subsection{Recurrence width}

The interval in \eqref{eq:5.6} gives only an a priori bound, since
different contributions to the same coefficient may cancel.  We
therefore distinguish the actual recurrence from this formal bound.

\begin{definition}
Let
\[
\ell=\min\{k:C_k\not\equiv0\},
\qquad
u=\max\{k:C_k\not\equiv0\}.
\]
The number
\[
\omega=u-\ell
\]
is called the \emph{recurrence width}.  We also define the
\emph{a priori recurrence width} by
\[
\omega_{\mathrm{ap}}
=
r+\max_{0\leq j\leq r}(d_j-j).
\]
\end{definition}

\begin{prop}
The recurrence width satisfies
\begin{equation}
\omega\leq\omega_{\mathrm{ap}}.
\label{eq:5.7}
\end{equation}
The inequality may be strict because of cancellations between
different contributions to the coefficients $C_k(n)$.
\end{prop}

\begin{proof}
By \eqref{eq:5.6},
\[
\ell\geq-r,
\qquad
u\leq\max_{0\leq j\leq r}(d_j-j).
\]
Hence
\[
\omega=u-\ell
\leq
r+\max_{0\leq j\leq r}(d_j-j).
\]
\end{proof}

\subsection{Finite reduction}

The finite recurrence leads to an effective reduction whenever its
highest-shift coefficient does not vanish.

\begin{thm}[Finite reduction]
Assume that
\[
C_u(n)\not\equiv0.
\]
Let
\[
\mathcal E=
\{n\in\mathbb Z_{\geq0}:C_u(n)=0\}.
\]
Then $\mathcal E$ is finite.  For every
$n\notin\mathcal E$,
\begin{equation}
I_{n+u}(x)
=
-\frac{1}{C_u(n)}
\sum_{k=\ell}^{u-1}
C_k(n)I_{n+k}(x)
+
\frac{E_n(x)}{C_u(n)}.
\label{eq:5.8}
\end{equation}
Consequently, after excluding the finitely many resonant indices,
all sufficiently large moments can be reduced recursively to a finite
set of initial moments and explicit boundary terms.
\end{thm}

\begin{proof}
Since $C_u(n)$ is a nonzero polynomial in $n$, it has only finitely many
integer zeros.  For $n\notin\mathcal E$, equation \eqref{eq:5.4} can be
solved for $I_{n+u}$, which gives \eqref{eq:5.8}.  All moments on the
right-hand side have smaller index than $n+u$.  Repeated application
therefore gives the asserted finite reduction.
\end{proof}

\begin{remark}
The elements of $\mathcal E$ will be called \emph{resonant indices}.
The recurrence itself remains valid at such indices, but the displayed
form \eqref{eq:5.8} cannot be used there.  A resonant moment may be
retained as an additional basic integral.
\end{remark}

\subsection{A matrix formulation}

The recurrence \eqref{eq:5.4} can be represented by a triangular
band system after the moments are ordered by increasing index.  For a
fixed target index $M$, only finitely many rows and columns are
required.  Let
\[
\mathcal M^{[M]}
\]
denote the corresponding finite truncation.

The resulting finite linear systems are triangular band systems for
which backward substitution is standard; see, for example,
\cite{quarteroni,stoer}.

If all diagonal entries of $\mathcal M^{[M]}$ are nonzero, then
\[
\det\mathcal M^{[M]}
=
\prod_i\mathcal M^{[M]}_{i,i}\neq0.
\]
Hence the truncation is nonsingular.  The reduction coefficients for
the target moment are obtained from the corresponding column by
backward substitution.  Thus, as in Section~2, no inverse of an
infinite matrix is required.

\subsection{Relation with the first-order case}

The first-order construction of Section~2 is recovered by taking
$r=1$ and
\[
B(x)K'(x)-A(x)K(x)=0.
\]
Thus
\[
a_1(x)=B(x),
\qquad
a_0(x)=-A(x).
\]
The resulting recurrence is equivalent to
\[
\frac{d}{dx}
\left(x^nB(x)K(x)\right)
=
\left[
nx^{n-1}B(x)
+
x^n(B'(x)+A(x))
\right]K(x).
\]
Hence the finite-band construction of Section~2 is the first-order
case of the present differential-equation formulation.

\subsection{A Bessel-function example}

Consider
\[
K(x)=J_2(x),
\]
where $J_2$ is the Bessel function of the first kind. It satisfies the
second-order differential equation
\[
x^2K''(x)+xK'(x)+(x^2-4)K(x)=0.
\]
Thus, in the notation of \eqref{eq:5.1},
\[
a_2(x)=x^2,
\qquad
a_1(x)=x,
\qquad
a_0(x)=x^2-4.
\]

For
\[
I_n(x)=\int x^nJ_2(x)\,dx,
\]
the general construction gives a finite recurrence. Indeed,
\[
\int x^{n+2}K''(x)\,dx
=
x^{n+2}K'(x)
-(n+2)x^{n+1}K(x)
+(n+2)(n+1)I_n(x),
\]
while
\[
\int x^{n+1}K'(x)\,dx
=
x^{n+1}K(x)
-(n+1)I_n(x).
\]
Substitution into the differential equation yields
\[
I_{n+2}(x)
+
\bigl((n+1)^2-4\bigr)I_n(x)
=
(n+1)x^{n+1}K(x)
-
x^{n+2}K'(x)
+C.
\]
Equivalently,
\begin{equation}
I_{n+2}(x)
=
(n+1)x^{n+1}J_2(x)
-
x^{n+2}J_2'(x)
+
\bigl(4-(n+1)^2\bigr)I_n(x)
+C.
\label{eq:5.bessel-recurrence}
\end{equation}

Thus the moments split into two finite reduction chains according to
parity. For example, repeated application of
\eqref{eq:5.bessel-recurrence} gives
\[
\begin{aligned}
I_6(x)
={}&
\left(5x^5-63x^3+105x\right)J_2(x)
\\
&+
\left(-x^6+21x^4-105x^2\right)J_2'(x)
+
315I_0(x)
+C.
\end{aligned}
\]
Hence
\[
\begin{aligned}
\int x^6J_2(x)\,dx
={}&
\left(5x^5-63x^3+105x\right)J_2(x)
\\
&+
\left(-x^6+21x^4-105x^2\right)J_2'(x)
+
315\int J_2(x)\,dx
+C.
\end{aligned}
\]
Using
\[
J_2'(x)=\frac{J_1(x)-J_3(x)}{2},
\]
the same reduction may also be written as
\[
\begin{aligned}
\int x^6J_2(x)\,dx
={}&
315\int J_2(x)\,dx
\\
&-
\frac12x^2\left(x^4-21x^2+105\right)
\left(J_1(x)-J_3(x)\right)
\\
&+
x\left(5x^4-63x^2+105\right)J_2(x)
+C.
\end{aligned}
\]
This example illustrates that the finite-band reduction is not
confined to first-order hyperexponential factors. A second-order
polynomial-coefficient differential equation for a classical special
function produces an explicit finite recurrence and reduces a
high-order moment to a finite basic family.

\section{Algorithmic reduction and complexity}

The finite-band structure developed above leads to an effective
algorithm for reducing a prescribed moment.  The computation can be
performed directly from the recurrence relation.  Equivalently, it can
be interpreted as backward substitution in a finite triangular band
system.

\subsection{Scalar reduction algorithm}

Consider a recurrence
\begin{equation}
\sum_{k=\ell}^{u} C_k(n)I_{n+k}(x)=E_n(x),
\label{eq:6.1}
\end{equation}
where
\[
C_u(n)\not\equiv0,
\]
and let
\[
\mathcal B=\{I_{m_1},\ldots,I_{m_s}\}
\]
be a chosen finite set of basic moments.

For a fixed target index $M$, we assume that
\[
C_u(n)\neq0
\]
for all indices used in the reduction.  The reduction is given by the
following algorithm.

\begin{algorithm}[h]
\caption{Direct scalar finite-band reduction}
\label{alg:scalar}
\begin{algorithmic}[1]
\Require Recurrence coefficients $C_k(n)$, boundary terms $E_n(x)$,
target index $M$, and a basic family
$\mathcal B=\{I_{m_1},\ldots,I_{m_s}\}$.
\Ensure A representation
\[
I_M=\sum_{j=1}^{s}r_{M,j}I_{m_j}+F_M(x).
\]

\State Initialize the representation of each basic moment $I_{m_j}$ by
the corresponding unit vector in $\mathbb R^s$.

\For{each non-basic index $m$ required to reach $M$, in increasing order}
    \State Set $n=m-u$.
    \If{$C_u(n)=0$}
        \State Declare $m$ resonant and retain $I_m$ as an additional
        basic moment.
    \Else
        \State Compute
        \[
        I_m=
        -\frac{1}{C_u(n)}
        \sum_{k=\ell}^{u-1}
        C_k(n)I_{n+k}
        +
        \frac{E_n(x)}{C_u(n)}.
        \]
        \State Replace the moments on the right-hand side by their
        previously computed representations.
    \EndIf
\EndFor

\State Collect the coefficients of the basic moments and the explicit
boundary terms.
\end{algorithmic}
\end{algorithm}

\begin{prop}
Assume that no resonance is encountered in
Algorithm~\ref{alg:scalar}.  Then the algorithm terminates after
finitely many steps and produces the unique representation of $I_M$
obtained from the recurrence \eqref{eq:6.1} and the prescribed basic
family.
\end{prop}

\begin{proof}
At every nonresonant step, equation \eqref{eq:6.1} can be solved for
the highest-index moment:
\[
I_m=
-\frac{1}{C_u(n)}
\sum_{k=\ell}^{u-1}
C_k(n)I_{n+k}
+
\frac{E_n(x)}{C_u(n)},
\qquad n=m-u.
\]
All moments on the right-hand side have index strictly smaller than
$m$.  Hence, when the moments are processed in increasing order, their
representations have already been computed or belong to the prescribed
basic family.  Since the target index $M$ is finite, only finitely many
steps are required.  The representation is uniquely determined at each
step because $C_u(n)\neq0$.
\end{proof}

\subsection{Finite matrix realization}

The direct recurrence algorithm has an equivalent matrix
interpretation.  For a fixed target index $M$, let
\[
\mathcal M^{[M]}
\]
denote the finite triangular band matrix obtained from the recurrence by
ordering the relations according to their largest moment index.  The
resulting finite linear systems are triangular band systems for which
backward substitution is standard; see, for example,
\cite{quarteroni,stoer}.

If all diagonal entries of $\mathcal M^{[M]}$ are nonzero, then
\[
\det\mathcal M^{[M]}
=
\prod_i\mathcal M^{[M]}_{i,i}\neq0.
\]
Consequently, $\mathcal M^{[M]}$ is nonsingular.  If
$v^{(M)}$ denotes the reduction column, then
\[
\mathcal M^{[M]}v^{(M)}=e_M.
\]
The entries of this column are computed by backward substitution:
\begin{equation}
v^{(M)}_M=
\frac{1}{\mathcal M^{[M]}_{M,M}},
\label{eq:6.2}
\end{equation}
and
\begin{equation}
v^{(M)}_i
=
-\frac{1}{\mathcal M^{[M]}_{i,i}}
\sum_{j=i+1}^{\min(i+h,M)}
\mathcal M^{[M]}_{i,j}v^{(M)}_j.
\label{eq:6.3}
\end{equation}
Thus the algorithm requires only a finite truncation.  In particular,
the inverse of the infinite recurrence matrix is never needed.

A vanishing diagonal entry of $\mathcal M^{[M]}$ is precisely the
matrix manifestation of a resonant index.  In this case the
corresponding moment has to be retained as an additional basic moment.

\subsection{Computational reduction example}

We illustrate the two equivalent implementations of the reduction
procedure on a simple case for which the result can also be checked by
hand.  Let
\[
K'(x)-K(x)=0,
\qquad\text{so that}\qquad K(x)=c e^x,
\]
and define
\[
I_n(x)=\int x^nK(x)\,dx.
\]
Multiplication of the differential equation by $x^n$ and one
integration by parts give
\begin{equation}
I_n=x^nK(x)-nI_{n-1}.
\label{eq:6.computational-recurrence}
\end{equation}
There are no resonant indices, and taking $I_0$ as the basic moment,
successive reduction gives
\[
I_1=xK-I_0,
\]
\[
I_2=(x^2-2x)K+2I_0,
\]
\[
I_3=(x^3-3x^2+6x)K-6I_0,
\]
and hence, for the target index $M=4$,
\begin{equation}
I_4=
\left(x^4-4x^3+12x^2-24x\right)K(x)+24I_0.
\label{eq:6.computational-I4}
\end{equation}

The accompanying implementation computes this reduction in two ways:
by direct recurrence substitution and by solving the corresponding
finite triangular band system.  For the target $I_4$ both procedures
return exactly the expression in \eqref{eq:6.computational-I4}; their
symbolic difference simplifies to zero.  This provides a small
reproducible check that the direct and finite-matrix realizations of the
algorithm agree on the same reduction problem.

\subsection{Arithmetic complexity}

We use the arithmetic operation model in which one addition,
subtraction, multiplication, or division in the coefficient field has
unit cost.

\begin{thm}[Scalar complexity]
Let $M$ be the target index and let
\[
\omega=u-\ell
\]
be the recurrence width.  Assume that the required recurrence
coefficients and diagonal inverses are available at unit cost.  Then
one reduction column can be computed in
\[
O(\omega M)
\]
arithmetic operations and
\[
O(M)
\]
storage.

In particular, for fixed recurrence width, the arithmetic complexity
is linear in the target index $M$.
\end{thm}

\begin{proof}
There are $O(M)$ backward-substitution steps, and at each step at most
$\omega$ terms contribute to the sum in \eqref{eq:6.3}.  Hence the
total number of arithmetic operations is $O(\omega M)$.  Storing one
reduction column requires $O(M)$ coefficients.
\end{proof}

The estimate concerns a single target moment.  Computing an entire
inverse matrix is a different problem and is not required by the
reduction procedure.

\subsection{Block-band complexity}

In the algebraic setting of Section~4, scalar recurrence coefficients
are replaced by $d\times d$ blocks.  Let $h$ denote the block bandwidth.
If ordinary matrix multiplication is used, multiplication of two
$d\times d$ matrices requires $O(d^3)$ arithmetic operations.

\begin{thm}[Block complexity]
Assume that the diagonal blocks are invertible.  Then one reduction
column of a block-band triangular system with target index $M$ can be
computed in
\[
O(hMd^3)
\]
arithmetic operations and
\[
O(Md^2)
\]
storage.
\end{thm}

\begin{proof}
There are $O(M)$ block rows and at most $h$ relevant blocks per row.
Each block operation has cost $O(d^3)$ in the standard matrix
multiplication model.  Hence the total cost is $O(hMd^3)$.  The
reduction column contains $O(M)$ blocks of size $d\times d$, requiring
$O(Md^2)$ storage.
\end{proof}

\subsection{Computational experiment}

To complement the arithmetic complexity estimate, we performed a
simple timing experiment for the scalar finite-band reduction with
fixed bandwidth $\omega=4$.  The benchmark measures the arithmetic
recurrence/back-substitution core of the reduction; symbolic
simplification routines such as \texttt{FullSimplify} are not included,
so that the experiment reflects the arithmetic-operation model used in
the preceding complexity theorem.

We considered target indices
\[
M=1000,\ 2000,\ 4000,\ 8000,\ 16000.
\]
For each value of $M$, the reported running time is the median of five
repetitions.  The results are shown in Table~\ref{tab:scalar-benchmark}.

\begin{table}[ht]
\centering
\par\vspace{8pt}
\begin{tabular}{r r r r}
\hline
$M$ & $\omega M$ & Median time (s) & Time/$M$ (s)\\
\hline
1000  & 4000  & 0.006129 & $6.129\times 10^{-6}$\\
2000  & 8000  & 0.009194 & $4.597\times 10^{-6}$\\
4000  & 16000 & $0.018627$ & $4.65675\times 10^{-6}$\\
8000  & 32000 & $0.040527$ & $5.065875\times 10^{-6}$\\
16000 & 64000 & $0.084455$ & $5.2784375\times 10^{-6}$\\
\hline
\end{tabular}

\par\vspace{12pt}

\caption{Scalar finite-band benchmark for fixed bandwidth $\omega=4$.
Each reported time is the median of five repetitions.}
\label{tab:scalar-benchmark}
\end{table}

The quotient $T(M)/M$ remains approximately constant over this range.
Moreover, a least-squares power-law fit of the form
\[
T(M)\sim cM^{\gamma}
\]
gives
\[
\gamma\approx0.97.
\]
Thus, within the tested range, the observed running time is essentially
linear in the target index $M$.  This numerical behaviour is consistent
with the theoretical arithmetic complexity $O(\omega M)$ for fixed
bandwidth.  The experiment is intended as a validation of the scaling
predicted by the arithmetic model, rather than as a machine-independent
performance estimate.

\begin{figure}[ht]
\centering
\includegraphics[width=0.72\textwidth]{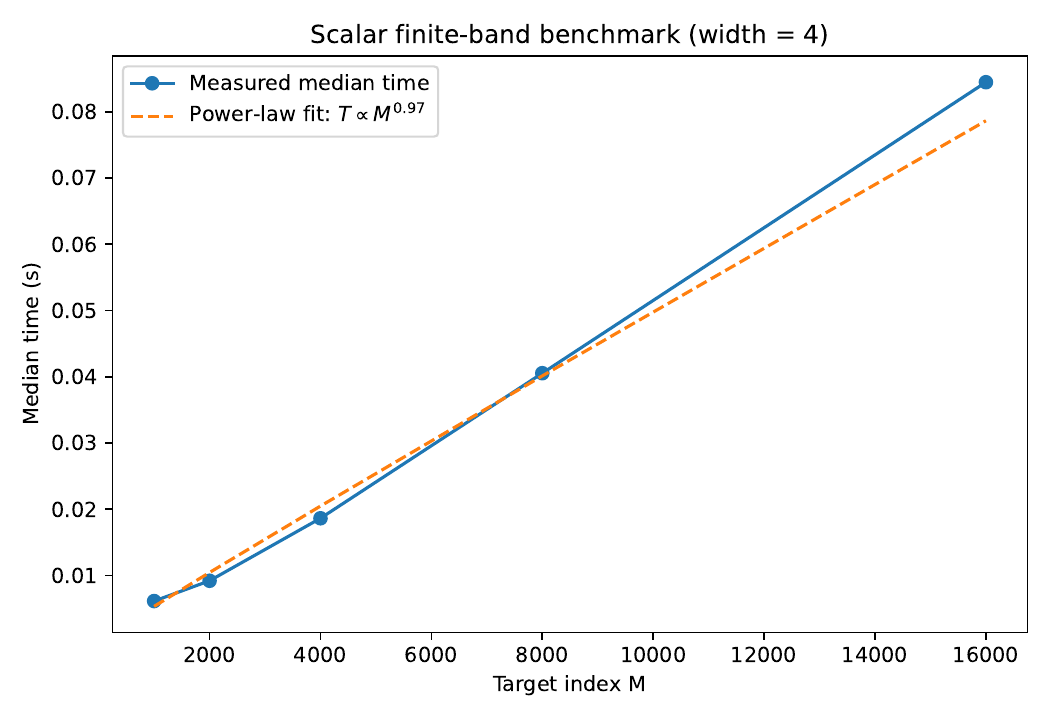}
\caption{Measured median running time for the scalar finite-band
benchmark with fixed bandwidth $\omega=4$, together with a power-law
fit $T(M)\sim cM^{0.97}$.}
\label{fig:scalar-benchmark}
\end{figure}

\subsection{Bit complexity and coefficient growth}

The arithmetic estimates above treat operations in the coefficient
field as unit-cost operations.  In symbolic applications, however,
the degrees and heights of the coefficients may grow with the target
index.  A complete bit-complexity analysis would therefore require
independent bounds on coefficient growth.

We do not pursue such bounds here.  The complexity results above
describe the structural cost of the finite-band reduction itself.

\subsection{Comparison of reduction representations}

The aim of the proposed method is reduction rather than the construction
of an elementary antiderivative.  In particular, the basic integrals
need not be elementary.  A meaningful computational comparison should
therefore be made at the level of the resulting reduction
representation.

For a target $I_M$, we record the profile
\[
\mathcal C_K(M)
=
\bigl(
\omega,\,
N_{\mathrm{basic}},\,
N_{\mathrm{nz}},\,
T,\,
S
\bigr),
\]
where $\omega$ is the recurrence width,
$N_{\mathrm{basic}}$ is the number of basic integrals,
$N_{\mathrm{nz}}$ is the number of nonzero reduction coefficients,
$T$ is the running time, and $S$ measures the size of the symbolic
representation.

This comparison does not require the basic integrals to be elementary.
The resulting representation
\[
I_M=
\sum_j r_{M,j}I_{m_j}+F_M(x)
\]
is useful precisely because it separates the explicitly computable
part from the finite set of basic integrals.

\subsection{Structured acceleration}

The entries of the finite-band matrices are generated by polynomial
coefficient sequences and have a prescribed dependence on the moment
index.  This suggests investigating whether additional Toeplitz-like,
displacement-rank, or related structure can be exploited for faster
reduction of families of target integrals.  We leave this question for
future work.
\section{Conclusion}

We have developed a finite-band framework for the reduction of families
of integrals arising from differential relations.  The construction
covers rational logarithmic derivatives, generalized real
Schwarz--Christoffel integrals, algebraic logarithmic derivatives, and
more general linear differential equations with polynomial
coefficients.  For a prescribed target integral, the reduction requires
only finite data and does not depend on an explicit evaluation of the
basic integrals, which may be non-elementary.

\section*{Software availability}

The reduction algorithms described in this paper are implemented in the
Mathematica package \emph{BandedReduction: Banded Reduction of Integrals for
Hyperexponential and Algebraic Functions} \cite{BandedReduction}, referred to
as \texttt{BandedReduction}. The package includes routines for the scalar
first-order reduction, the generalized Schwarz--Christoffel specialization,
algebraic block-band reduction, and reduction derived from linear differential
equations with polynomial coefficients. It also contains regression tests and
a compact user-facing report for the resulting reductions. The software is
archived on Zenodo under DOI 10.5281/zenodo.22980629.

\end{document}